\documentclass[10pt]{article}

\usepackage{amsmath}
\usepackage{amsfonts}
\usepackage{amsthm}
\usepackage[english]{babel} 
\usepackage{hyperref}

\usepackage{amssymb}
\usepackage{graphicx}        
 \usepackage{cite}                            
\usepackage{subfigure}
\usepackage[table]{xcolor}
\usepackage{float}
\usepackage{array}
\usepackage[scale=0.8]{geometry}

\newtheorem{theorem}{Theorem}[section]

\newtheorem{proposition}{Proposition}[section]

\def\derpar#1#2{\frac{\partial#1}{\partial#2}}
\newcommand{\ens}[1]{\mathbb{#1}}
\newcommand{\bb}{\hat \beta}

\def\R{\mathbb R}
\newcommand{\f}{\hat f}

\def\P{\mathcal P}
\def\proj{\mathcal P}

\def\D{\mathcal D}
\def\B{\mathcal B}

\def\poly{\mathbb P}
\def\f{\hat f}
\def\LL{\mathcal L}
\newcommand{\Q}{\mathcal{Q}}

\def\Ball{ {\cal B}}

\DeclareMathOperator{\Span}{span}

\def\be{\begin{equation}}
\def\ee{\end{equation}}
\def\bea{\begin{eqnarray}}
\def\eea{\end{eqnarray}}
\def\beas{\begin{eqnarray*}}
\def\eeas{\end{eqnarray*}}

\title{On the stability of equilibrium preserving spectral methods for the homogeneous Boltzmann equation}
\author{Lorenzo Pareschi\footnote{Department of Mathematics \& Computer Science, University of Ferrara, Via Machiavelli 30, Ferrara, 44121, Italy(\texttt{lorenzo.pareschi@unife.it}).} and Thomas Rey\footnote{Univ. Lille, CNRS, UMR 8524, Inria - Laboratoire Paul Painlevé F-59000 Lille, France (\texttt{thomas.rey@univ-lille.fr})}} 
\date{}

\begin{document}
\maketitle
\begin{abstract} Spectral methods, thanks to the high accuracy and the possibility to use fast algorithms, represent an effective way to approximate the Boltzmann collision operator. On the other hand, the loss of some local invariants leads to the wrong long time behavior. A way to overcome this drawback, without sacrificing spectral accuracy, has been proposed recently with the construction of equilibrium preserving spectral methods. Despite the ability to capture the steady state with arbitrary accuracy, the theoretical properties of the method have never been studied in details. In this paper, using the perturbation argument developed by Filbet and Mouhot for the homogeneous Boltzmann equation, we prove stability, convergence and spectrally accurate long time behavior of the equilibrium preserving approach.    \\[1em]
    \textsc{Keywords:} Boltzmann equation, Fourier-Galerkin spectral method, steady-state preserving, micro-macro decomposition, local Maxwellian, stability\\[.5em]
    \textsc{2010 Mathematics Subject Classification:} 
     65N35, 
     76P05, 
  82C40 
  
\end{abstract}

\section{Introduction}
Spectral methods for the Boltzmann equation are nowadays rather popular in the deterministic numerical solution of the Boltzmann equation (see \cite{DimarcoPareschi15} for an introduction and \cite{wu2013deterministic,gamba2017fast,cai2018entropic,jaiswal2019fast,hu2020new} for some recent developments). Their main features are the spectral accuracy for smooth solutions and the possibility to use fast algorithms that mitigate the curse of dimensionality \cite{PR00, MP06}. On the other hand, being based on a Fourier-Galerkin approach, they typically lead to the loss of most physical properties of the Boltzmann equation, namely positivity, conservations, entropy dissipation and, as a consequence, long time behavior \cite{DimarcoPareschi15}. 
Typically, approaches based on smoothing to recover positivity or renormalization to force moments conservation lead to the loss of spectral accuracy and do not guarantee the correct long time behavior \cite{PR00b}. A way to overcome some of these drawbacks without sacrificing spectral accuracy has been proposed recently in \cite{FilbetPareschiRey:2015,PareschiRey:2016}, where steady state preserving spectral methods have been constructed. Despite their ability to capture the steady state with arbitrary accuracy, stability and convergence of the methods have never been studied in details. In this paper, we fill that gap by showing that the general stability arguments developed for the classical spectral methods by Filbet and Mouhot in \cite{FM11} can be extended to their corresponding equilibrium preserving approaches.

\subsection{The Boltzmann equation}
The homogeneous Boltzmann equation describes the behavior of a dilute gas of
particles when the only interactions taken into account are binary
elastic collisions. It reads for  $v \in \R^{d}$, $d \leq 3$
  \begin{equation*}
  \derpar{f}{t} = \Q(f,f)
  \end{equation*}
where $f(t,v)$ is the time-dependent particle distribution
function in the phase space. The Boltzmann collision operator $\Q$
is a quadratic operator local in $t$:
  \begin{equation}\label{eq:mp:Q}
  \Q (f,f)(v) = \int_{\R^{d} \times\ens{S}^{d-1}} B(|v-v_*|,\cos \theta) \,
  \left( f'_* f' - f_* f \right) \, dv_* \, d\sigma.
  \end{equation}
In~\eqref{eq:mp:Q} we used the shorthand $f = f(v)$, $f_* = f(v_*)$,
$f ^{'} = f(v')$, $f_* ^{'} = f(v_* ^{'})$. The velocities of the
colliding pairs $(v,v_*)$ and $(v',v'_*)$ can be parametrized as
  \begin{equation*}
  v' = \frac{v+v_*}{2} + \frac{|v-v_*|}{2} \sigma, \qquad
  v'_* = \frac{v+v^*}{2} - \frac{|v-v_*|}{2} \sigma.
  \end{equation*}
The collision kernel $B$ is a non-negative function which by
physical arguments of invariance only depends on $|v-v_*|$ and
$\cos \theta = {\hat g} \cdot \sigma$ (where ${\hat g} =
(v-v_*)/|v-v_*|$).
It characterizes the
details of the binary interactions, and and will take in this note the following variable hard spheres form
\begin{equation}
\label{defVHSKernel}
B(|v-v_*|,\cos\theta)=|v-v_*|^\lambda b(\cos(\theta),
\end{equation}
where $\lambda \in [0,1]$ and $b \in L^1(\mathbb{S}^{d-1})$.

Boltzmann's collision operator has the fundamental properties of
conserving mass, momentum and energy
  \begin{equation*}
  \int_{\R^d} \Q(f,f)\phi(v) \,dv = 0, \quad
  \phi(v)=1,v_1, \ldots, v_d,|v|^2
 \end{equation*}
and satisfies the well-known Boltzmann's $H$ theorem
  \begin{equation*} 
  \frac{d}{dt} \int_{\R^d} f \log f \, dv = \int_{\R^d} \Q(f,f) \log(f) \, dv \leq 0.
  \end{equation*}
Boltzmann's $H$ theorem implies that any equilibrium
distribution function has the form of a locally Maxwellian distribution
  \begin{equation*}
  M(\rho,u,T)(v)=\frac{\rho}{(2\pi T)^{d/2}}
  \exp \left\{ - \frac{\vert u - v \vert^2} {2T} \right\},
  \end{equation*}
where $\rho,\,u,\,T$ are the density, mean velocity
and temperature of the gas
  \begin{equation*}
  \rho = \int_{\R^d}f(v) \, dv, \quad
  u = \frac{1}{\rho}\int_{\R^d} v f(v) \, dv, \quad
  T = \frac{1}{d\rho} \int_{\R^d}\vert u - v \vert^2f(v) \, dv.
  \end{equation*}
For further details on the physical background and derivation of
the Boltzmann equation we refer to~\cite{CIP:94}.

\subsection{The Fourier spectral method for the Boltzmann collision operator}
To simplify notations we shall derive the spectral method in the classical Fourier-Galerkin setting introduced in \cite{PR00}, similarly it can be extended to the representation used in \cite{MP06} for the derivation of fast algorithms. Thus, we perform the usual periodization in a bounded domain of the operators $\LL$ and $\Q$ and denote by $\LL^R$ and $\Q^R$ the operators with truncation on the relative velocity on $\Ball_0(2R)$. 

Let us first set up the mathematical framework of our analysis. For any
$t \geq 0$, $f_N(t,v)$ is a trigonometric polynomial of degree $N$ in
$v$, i.e. $f_N(t) \in \poly^N$ where
\[
\mathbb P^N = \Span\left\{e^{ik\cdot v}\,|\, -N \leq k_j \leq N,\, j=1,\ldots,d
\right\}.
\]
Moreover, let $\proj_N : L^2([-\pi,\pi]^3) \rightarrow \poly^N$ be the
orthogonal projection upon $\poly^N$ in the inner product of
$L^2([-\pi,\pi]^3)$ 
\[
<f-\proj_N f,\phi>=0,\qquad \forall\,\, \phi\,\in\,\poly^N.
\]
We denote the $L^2$-norm by
\[
||f||_2 = (< f, f>)^{1/2}.
\]
With this definition $\proj_N f=f_N$, where $f_N$ is the truncated
Fourier series of $f$ given by 
\begin{equation*}
f_N(v) = \sum_{k=-N}^N \f_k e^{i k \cdot v}, \qquad \f_k = \frac{1}{(2\pi)^d}\int_{[-\pi,\pi]^d} f(v)
e^{-i k \cdot v }\,dv.
\end{equation*}
Using multi-index notations, we can compute $Q_N^R(f_N,f_N) := \proj_N Q^R(f_N,f_N)$ as
\begin{equation}
Q_N^R(f_N,f_N) = \sum_{k=-N}^N \left(\sum_{\substack{l,m=-N \\l+m=k}}^N \f_l\,\f_m
\bb(l,m)\right)e^{i k \cdot v},\quad k=-N,\ldots,N,
\label{eq:CF1}
\end{equation}
where the {Boltzmann kernel modes} $\bb(l,m)=\B(l,m)-\B(m,m)$ are given by
\begin{equation}
\B(l,m) = \int_{\Ball_0(2\lambda\pi)}\int_{\mathbb{S}^{d-1}} 
B(\cos \theta,|q|) e^{-i(l\cdot q^++m\cdot q^-)}\,d\omega\,dq. \label{eq:KM}
\end{equation}
In this last identity, $q=v-v_*$  and $q^+, q^-$ are given by
  \[
    q^+ = \frac12\left (q+|q|\omega\right ), \quad q^- = \frac12\left (q-|q|\omega\right ).
  \] 

\section{The equilibrium preserving spectral method}
	A major problem associated with deterministic methods is that the velocity space is approximated by a finite region. On the other hand, even starting from a compactly supported function in velocity, by the action of the collision term the solution becomes immediately positive in the whole velocity space. In particular, the local Maxwellian equilibrium states are characterized by exponential functions defined on the whole velocity space.
Moreover, these gaussian functions are obviously not contained in the space of trigonometric polynomials $\mathbb{P}^N$.
The main idea in the derivation of equilibrium preserving spectral methods is to perform a change of variables such that the new equilibrium belongs to the approximation space\cite{FilbetPareschiRey:2015,PareschiRey:2016}. In our case, we will see that $0$ is a good candidate.

Let us start with the decomposition 
\be
f=M+g,
\label{eqs5:micmac}
\ee  
with $M$ the local Maxwellian equilibrium and $g$ such that $\int_{\R^d} g\,\Phi\,dv=0$. When
inserted into the Boltzmann collision operator, the decomposition \eqref{eqs5:micmac} gives
\be
Q(f,f)=\LL(M,g)+Q(g,g),
\label{eq:decom}
\ee
where $\LL(M,g)=Q(g,M)+Q(M,g)$ is a linear operator and we used the fact that 
\be
Q(M,M)=0.
\label{eq:steady}
\ee
There are two major features in the decomposition (\ref{eq:decom}): 
\begin{enumerate}
\item it embeds the identity (\ref{eq:steady}); 
\item the steady state of (\ref{eq:decom}) is given by $g=0$.
\end{enumerate} 
{This type of micro-macro decomposition has been used \emph{e.g.} in \cite{JinShi:2009, LemMieu:2008, BcHR:2019} to develop numerical methods which preserves asymptotic behaviors of some kinetic models.}

To illustrate the equilibrium preserving spectral method, let us consider now the space homogeneous Boltzmann equation  that we rewrite using the micro-macro decomposition as
\be
\left\{
\begin{aligned}
\frac{\partial g}{\partial t} &= \LL(M,g)+Q(g,g),\\
f&=M+g.
\end{aligned}
\label{eq:micmac}
\right.
\ee
Denoting  $M_N=\proj_N M$ and $g_N=\proj_N g$,  we then write the Fourier-Galerkin approximation of the micro-macro equation \eqref{eq:micmac} 
\be
\left\{
\begin{aligned}
\frac{\partial g_N}{\partial t}&=
\LL_N(M_N,g_N)+Q_N(g_N,g_N),\\
f_N&=M_N+g_N,
\end{aligned}
\label{eq:specc}
\right.
\ee
where
$\LL_N(M_N,g_N)=\proj_N \LL(M_N,g_N)$ and  $Q_N(g_N,g_N)=\proj_N Q^R(g_N,g_N)$. 

It is immediate to show that 
\begin{proposition}
The function $g_N \equiv 0$ is an admissible local equilibrium of the scheme (\ref{eq:specc}) and therefore $f_N=M_N$ is a local equilibrium state. 
\end{proposition}
Moreover it was proved in \cite{FilbetPareschiRey:2015} that the equilibrium preserving spectral method has the same spectral consistency property that the underlying spectral method:
\begin{theorem}
Let $f \in H_p^r([-\pi,\pi]^3)$, $r\geq 1$ then there exists $C>0$ such that
\begin{eqnarray}
\nonumber
\|Q(f,f)-\LL_N^\lambda(M_N,g_N)-Q_N(g_N,g_N)\|_{L^2} &\leq& \frac{C}{N^r} \left(\|f\|_{H^r} + \|M\|_{H^r}\right.\\[-.25cm]
\\[-.25cm]
\nonumber
&&\left. +
\|Q(f_N,f_N)\|_{H^r}+\|Q(M_N,M_N)\|_{H^r}\right).
\end{eqnarray}
\label{th:news}
\end{theorem}  

\section{Stability of the equilibrium preserving spectral method}

Consider the equilibrium preserving method written in the form
\be
\left\{
\begin{aligned}
\frac{\partial f_N}{\partial t } &= \P_N Q(f_N,f_N)-\P_N Q(M_N,M_N) =\P_N Q(f_N+M_N,f_N-M_N)\\
f_N(v,0)&=\P_N f_0(v)
\label{eq:spece}
\end{aligned} \right.
\ee
where $M_N = \P_N M$ does not depend on time. Our main result is summarized below. 

	\begin{theorem}
		\label{thm:stabilityEPBoltzmann}
		Let us consider a nonnegative initial condition $f_0 \in H^k([\pi,\pi]^d)$ for $k > d/2$. There exists $N_0 \in \mathbb N$  such that for all $N \geq N_0$:
		\begin{enumerate}
			\item There is a unique global smooth solution $f_N$ to the problem \eqref{eq:spece};
			\item for any $r < k$, there exists $C_r >0$ such that 
			\[ \|f_N^c(t,\cdot)\|_{H^r} \leq C_r; \]
			\item $f_N$ converges to a solution $f$ of equation \eqref{eq:mp:Q} with spectral accuracy, uniformly in time;
			\item there exists $C, \lambda > 0$ depending on $f_0$ such that 
				\[ \| g_N(t,\cdot)\|_{L^1} \leq C e^{-\lambda t}, \quad \text{or equivalently,} \quad \| f_N(t,\cdot) - M_N\|_{L^1} \leq C e^{-\lambda t} .\]
		\end{enumerate}
	\end{theorem}

\begin{proof}
We want to apply the Theorem 3.1 of \cite{FM11}. We rewrite \eqref{eq:spece} as a perturbed Boltzmann equation
\be
\left \{\begin{aligned}
\frac{\partial f_N}{\partial t } &= Q(f_N,f_N)+P_N(f_N,M_N)\\
f_N(v,0)&=\P_N f_0(v)
\label{eq:spece2}
\end{aligned} \right.
\ee
with the perturbation $P_N$ defined as
\be
P_N(f_N,M_N)= \P_N Q(f_N+M_N,f_N-M_N)- Q(f_N,f_N).
\label{eq:pert}
\ee
Our goal is to prove that this perturbation fits the framework developed in \cite{FM11} to study the stability and large time behavior of spectral methods, namely that \textbf{($H_1$)} it \emph{preserves the mass}; \textbf{($H_2$)} is \emph{smooth}; \textbf{($H_3$)} is \emph{spectrally small}. Let us first notice the following identities:

\[
\int_{\D_L}P_N(f_N,M_N)\,dv = 0,\qquad P_N (M_N,M_N)=-Q(M_N,M_N).
\]
In particular, the only equilibrium distributions to \eqref{eq:spece} or \eqref{eq:spece2} are the corresponding projections of the Maxwellian equilibrium distribution $M_N$. Moreover, this also proves mass preservation, namely hypothesis \textbf{($H_1$)}.

For a nonnegative function $f\in H^k$, we have for $r\in[0,k]$ and $M=M[f]$
\[
\| P_N(f,M) \|_{H^r} \leq \|\P_N Q(f+M,f-M)\|_{H^r}+\|Q(f,f)\|_{H^r}.
\]
Now since $Q(M,M)=0$, one gets from the regularity Lemma 4.1 of \cite{FM11} and Lemma 5.2 of \cite{PR00}
\[
\|Q(f,f)\|^2_{H^r} =\|Q(f+M,f-M)\|^2_{H^r} \leq C \|f+M\|_{L^1}^2 \|f-M\|_{H^r}^2 
\]
and
\[
\|\P_N Q(f+M,f-M)\|^2_{H^r} \leq \|Q(f+M,f-M)\|^2_{H^r}\leq C \|f+M\|_{L^1}\|f-M\|^2_{H^r} 
\]
so that get the smoothness hypothesis \textbf{($H_2$)}
\be
\| P_N(f,M) \|_{H^r} \leq C \|f+M\|_{L^1}^2 \|f-M\|_{H^r}^2.
\ee
Similarly, we also have the spectral smallness \textbf{($H_3$)}, namely:
\be
\| P_N(f,M) \|_{H^r} \leq C \|f+M\|_{L^1}^2 \frac{\|f-M\|_{H^k}^2}{N^{k-r}}.
\ee
Finally, for the global stability, note that from
\[
\|Q(f+M,f-M)\|\leq C \|f+M\|_{L^1} \|f-M\|_{L^2_p}
\]
we get
\beas
\frac12 \frac{d}{dt}\|f-M\|^2_{L^2_p} &\leq& \|Q(f+M,f-M)+P_N(f,M)\|_{L^2_p} \|f-M\|_{L^2_p} \\
&\leq& \|f+M\|_{L^1} \|f-M\|^2_{L^2_p}
\eeas
Now assume $\| f+M\|_{L^1} \leq K$ we get by Gronwall inequality that there exist a constant $C_0(K)$ s.t.
\[
\|f-M\|^2_{L^2_p} \leq C_0(K).
\]
Similarly we obtain that there exist a constant $C_r(K)$ such that 
\[
\|f-M\|^2_{H^r} \leq C_r(K).
\]
Gathering this and hypotheses \textbf{(${H_1}-{H_2}-{H_3}$)}, one obtains using the Theorem 3.1 of \cite{FM11} our result.

\end{proof}
\paragraph{Acknowledgment.}
TR was partially funded by Labex CEMPI (ANR-11-LABX-0007-01) and ANR Project MoHyCon (ANR-17-CE40-0027-01). LP acknowledge the partial support of MIUR-PRIN Project 2017, No. 2017KKJP4X Innovative numerical methods for evolutionary partial differential equations and applications. 

\bibliographystyle{acm}
\bibliography{biblioPR}

\end{document}